\documentclass[12pt]{article}

\usepackage{graphicx,amsthm}
\usepackage{amsmath,amssymb}
\usepackage{blindtext}

\newtheorem{theorem}{Theorem}[section]
\newtheorem{corollary}[theorem]{Corollary}

\newtheorem{proposition}[theorem]{Proposition}

\newcommand{\f}{\operatorname}

\title{The Generalized Weighted Lindley Distribution:  Properties, Estimation and Applications }
\author{P. L. RAMOS,\thanks{ Email: pedrolramos@usp.br}  
\ \ \ \ F. LOUZADA\thanks{ Email: louzada@icmc.usp.br} \ \  \  \\ \\ Institute of Mathematical Science and Computing \\ Universidade de Sao Paulo, Sao Carlos-SP, Brazil } 
\date{January/2016}
\begin{document}

\maketitle

\begin{abstract}
In this paper, we proposed a new lifetime distribution namely generalized weighted Lindley (GLW) distribution. The GLW distribution is a useful generalization of the weighted Lindley distribution, which accommodates increasing, decreasing, decreasing-increasing-decreasing, bathtub, or unimodal hazard functions, making the GWL distribution a flexible model for reliability data. A significant account of mathematical properties of the new distribution are presented. Different estimation procedures are also given such as, maximum likelihood estimators, method of moments, ordinary and weighted least-squares, percentile, maximum product of spacings and minimum distance estimators. The different estimators are compared by an extensive numerical simulations. Finally, we analyze two data sets for illustrative purposes, proving that the GWL outperform several usual three parameters lifetime distributions.

\end{abstract}

\section{Introduction}

In recent years, several new extensions of the exponential distribution has been introduced in the literature for describing real problems. Ghitany et al. (2008) investigated different properties of the Lindley distribution and outlined that in many cases the Lindley distribution is a better model than one based on the exponential distribution. Since then, many generalizations of the Lindley distribution have been introduced, such as generalized Lindley (Zalerzadeh and Dolati, 2009), extended Lindley (Bakouch et al., 2012), exponential Poisson Lindley (Barreto-Souza and Bakouch, 2013), Power Lindley (Ghitany et al., 2013) distribution, among others.

Ghitany et al. (2011) introduced a new class of weighted Lindley (WL) distribution adding more flexibility to the Lindley distribution. Let $T$ be a random variable with a WL distribution the probability density function (p.d.f) is given by
\begin{equation}\label{fdpwl} 
f(t|\lambda,\phi)=\frac{\lambda^{\phi+1}}{(\lambda+\phi)\Gamma(\phi)}t^{\phi-1}(1+t)e^{-\lambda t},
 \end{equation}
for all $t>0$ , $\phi>0$ and $\lambda>0$ and $\Gamma(\phi)=\int_{0}^{\infty}{e^{-x}x^{\phi-1}dx}$ is known as gamma function. One of its peculiarities is that the hazard function can has an increasing $(\phi\geq 1)$ or bathtub $(0<\phi<1)$ shape. Different properties of this model and estimation methods were studied by Mazucheli et al. (2013), Ali (2013), Wang (2015), Al-Mutairi et al. (2015) among others.

In this paper, a new lifetime distribution family is proposed, which is an direct generalization of the weighted Lindley distribution. The p.d.f is given by
\begin{equation}\label{fdpgwl} 
f(t|\phi,\lambda,\alpha)=\frac{\alpha\lambda^{\alpha\phi}}{(\lambda+\phi)\Gamma(\phi)}t^{\alpha\phi-1}(\lambda+(\lambda t)^{\alpha})e^{-(\lambda t)^\alpha},
 \end{equation}
for all $t>0$,  $\phi>0, \lambda>0$ and $\alpha>0$. Important probability distributions can be obtained from the GWL distribution as the weighted Lindley distribution ($\alpha=1$) , Power Lindley distribution ($\phi=1$) and the Lindley distribution ($\phi=1$ and $\alpha=1$). Due this relationship, such model could also be named as weighted power Lindley or generalized power Lindley distribution. We present a proof that this model has different forms of the hazard function, such as: increasing, decreasing, bathtub, unimodal or decreasing-increasing-decreasing shape, making the GWL distribution a flexible model for reliability data. Moreover, a significant account of mathematical properties of the new distribution are provided.

The inferential procedures of the parameters of the GLW distribution are presented considering different estimation methods such as: maximum likelihood estimators (MLE), methods of moments (ME), ordinary least-squares estimation (OLSE), weighted least-squares estimation (WLSE), maximum product of spacings (MPS), Cramer-von Mises type minimum distance (CME), Anderson-Darling (ADE) and Right-tail Anderson-Darling (RADE).  We compare the performances of the such different methods using extensive numerical simulations. Finally, we analyze two data sets for illustrative purposes, proving that the GWL outperform several usual three parameters lifetime distributions such as: the generalized Gamma distribution (Stacy, 1962), the generalized Weibull (GW) distribution (Mudholkar et al., 1996), the generalized exponential-Poisson (GEP) distribution (Barreto-Souza and Cribatari-Neto, 2009) and the exponentiated Weibull (EP) distribution (Mudholkar et al., 1995).

The paper is organized as follows. In Section 2, we provide a significant account of mathematical properties of the new distribution. In Section 3,  we discuss the eight estimation methods considered in this paper. In Section 4 a simulation study is presented in order to identify the most efficient procedure. In Section 5 the methodology is illustrated in two real data sets. Some final comments are presented in Section 6.

\section{Generalized Weighted Lindley distribution}

The Generalized weighted Lindley distribution (\ref{fdpgwl}) can be expressed as a two-component mixture
\begin{equation*}\label{fdpwl2} 
f(t|\phi,\lambda,\alpha)= pf_1(t|\phi,\lambda,\alpha)+(1-p)f_2(t|\phi,\lambda,\alpha)
 \end{equation*}
where  $p=\lambda/(\lambda+\phi)$ and  $T_j\sim\f{GG}(\phi+j-1,\lambda,\alpha)$, for $j=1,2$, i.e, $f_j(t|\lambda,\phi)$ has Generalized Gamma distribution, given by
\begin{equation}\label{fdpgg} 
f_j(t|\phi,\lambda,\alpha)=\frac{\alpha}{\Gamma(\phi+j-1)}\lambda^{\alpha(\phi+j-1)}t^{\alpha(\phi+j-1)-1}e^{-(\lambda t)^\alpha}.
 \end{equation}
 
The behaviours of the p.d.f. (\ref{fdpgwl}) when $t\rightarrow0$ and $t\rightarrow\infty$ are, respectively, given by
\begin{equation*}
f(0)=
\begin{cases}
 \infty, & \text{if }\alpha\phi<1 \\
 \dfrac{\alpha\lambda^2}{(\lambda+\phi)\Gamma(\phi)}, & \text{if }\alpha\phi=1 \\
 0, & \text{if }\alpha\phi>1 
\end{cases}
,\ \ \ \ \ \ \ \ f(\infty)= 0.
\end{equation*}

Figure \ref{frisgwl} gives examples of the shapes of the density function for different values of $\phi, \lambda$ and $\alpha$.
\begin{figure}[!htb]
\centering
\includegraphics[scale=0.53]{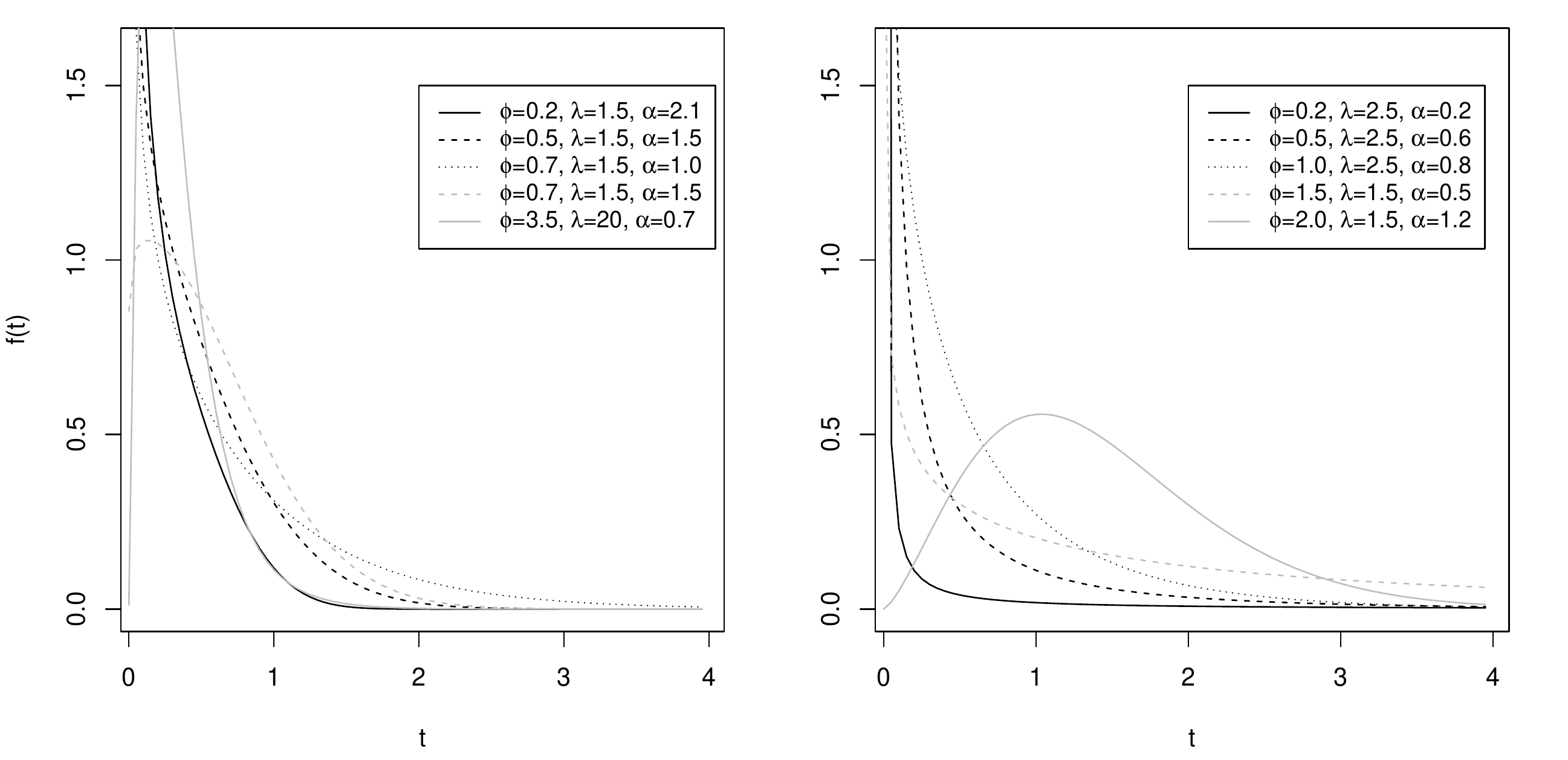}
\caption{Density function shapes for GWL distribution considering different values of $\phi, \lambda$ and $\alpha$.}\label{frisgwl}
\end{figure}

The cumulative distribution function from the GWL distribution is given by
\begin{equation}\label{densagwl}
F(t|\phi,\lambda,\alpha)=\frac{\gamma\left[\phi,(\lambda t)^{\alpha}\right](\lambda+\phi)-(\lambda t)^{\alpha\phi}e^{-(\lambda t)^\alpha}}{(\lambda+\phi)\Gamma(\phi)}\ .
\end{equation}
where  $\gamma[y,x]=\int_{0}^{x}{w^{y-1}e^{-w}}dw$ is the lower incomplete gamma function.

\subsection{Moments}

Many important features and properties of a distribution can be obtained through its moments, such as mean, variance, kurtosis and skewness. In this section, we present some important moments, such as the moment generating function, r-th moment, r-th central moment among others.

\begin{theorem} For the random variable $T$ with $\f{GWL}$ distribution,
the moment generating function is given by
\begin{equation}
M_X(t)=\sum_{r=0}^{\infty} \frac{t^r}{\lambda^r r!}\frac{\left(\frac{r}{\alpha}+\phi+\lambda\right)\Gamma(\frac{r}{\alpha}+\phi)}{(\lambda+\phi)\Gamma(\phi)}  .
\end{equation}
\end{theorem}
\begin{proof}Note that, the moment generating function from GG distribution (\ref{fdpgg}) is given by
\begin{equation*}
M_{X,j}(t)=\sum_{r=0}^{\infty} \frac{t^r}{ r!}\frac{\Gamma(\frac{r}{\alpha}+\phi+j-1)}{\lambda^r\Gamma(\phi+j-1)}.
\end{equation*}

Therefore, as the GWL (\ref{fdpgwl}) distribution can be expressed as a two-component mixture, we have
\begin{equation*}
\begin{aligned}
M_X(t)&= E\left[e^{tX} \right]= \int_{0}^{\infty}e^{tx}f(x|\phi,\lambda,\alpha)dx 
= pM_{X,1}(t)+(1-p)M_{X,2}(t) \\& = \frac{\lambda}{(\lambda+\phi)}\sum_{r=0}^{\infty} \frac{t^r}{ r!}\frac{\Gamma(\frac{r}{\alpha}+\phi)}{\lambda^r\Gamma(\phi)}+\frac{\phi}{(\lambda+\phi)}\sum_{r=0}^{\infty} \frac{t^r}{ r!}\frac{\Gamma(\frac{r}{\alpha}+\phi+1)}{\lambda^r\Gamma(\phi+1)} \\& = \frac{1}{(\lambda+\phi)}\sum_{r=0}^{\infty} \frac{t^r}{ r!}\frac{\lambda\Gamma(\frac{r}{\alpha}+\phi)}{\lambda^r\Gamma(\phi)}+\frac{1}{(\lambda+\phi)}\sum_{r=0}^{\infty} \frac{t^r}{ r!}\frac{\left(\frac{r}{\alpha}+\phi\right)\Gamma(\frac{r}{\alpha}+\phi)}{\lambda^r\Gamma(\phi)} \\& = \sum_{r=0}^{\infty} \frac{t^r}{\lambda^r r!}\frac{\left(\frac{r}{\alpha}+\phi+\lambda\right)\Gamma(\frac{r}{\alpha}+\phi)}{(\lambda+\phi)\Gamma(\phi)} .
\end{aligned}
\end{equation*}

\end{proof}

\begin{corollary} For the random variable $T$ with $\f{GWL}$ distribution,
the r-th moment is given by
\begin{equation}\label{rmgwl}
\mu_r= E[T^r]=\frac{\left(\frac{r}{\alpha}+\phi+\lambda\right)\Gamma(\frac{r}{\alpha}+\phi)}{(\lambda+\phi)\lambda^r \Gamma(\phi)} . 
\end{equation}
\end{corollary}
\begin{proof} From the literature $\mu_r=M_{X}^{(r)}(0)=\frac{d^n M_X(0)}{dt^n}$ and the result follows.
\end{proof}

\begin{corollary} For the random variable $T$ with $\f{GWL}$ distribution,
the r-th central moment is given by
\begin{equation}\label{rcmgwlp}
\begin{aligned} 
M_r&= E[T-\mu]^r= \sum_{i=0}^{r}\binom{r}{i}(-\mu)^{r-i}E[T^i] \\ &= \sum_{i=0}^{r}\binom{r}{i}\left(-\frac{\left(\frac{1}{\alpha}+\phi+\lambda \right)\Gamma\left(\frac{1}{\alpha}+\phi \right)}{\lambda(\lambda+\phi)\Gamma(\phi)}\right)^{r-i}\frac{\left(\frac{i}{\alpha}+\phi+\lambda\right)\Gamma(\frac{i}{\alpha}+\phi)}{(\lambda+\phi)\lambda^i \Gamma(\phi)} 
\end{aligned} 
\end{equation}
\end{corollary}

\begin{corollary} A random variable $T$ with $\f{GWL}$ distribution,
has the mean and variance  given by
\begin{equation}\label{meangwl} 
\mu=\frac{\left(\frac{1}{\alpha}+\phi+\lambda \right)\Gamma\left(\frac{1}{\alpha}+\phi \right)}{\lambda(\lambda+\phi)\Gamma(\phi)}, 
\end{equation}
\begin{equation}\label{vargwl} 
 \sigma^2=\frac{\lambda(\lambda+\phi)\left(\frac{2}{\alpha}+\phi+\lambda \right)\Gamma\left(\frac{2}{\alpha}+\phi \right)-\left(\frac{1}{\alpha}+\phi+\lambda \right)^2\Gamma\left(\frac{1}{\alpha}+\phi \right)^2}{\lambda^2(\lambda+\phi)^2\Gamma(\phi)^2}.
\end{equation}
\end{corollary}
\begin{proof} From (\ref{rmgwl}) and considering $r=1$ follows $\mu_1=\mu$. The second result follows from (\ref{rcmgwlp}) considering $r=2$ and with some algebra follow the results.\end{proof}

Different type of moments can be easily achieved for GWL distribution, one in particular, that has play a important role in information theory is given by
\begin{equation}\label{lrmgwl}
E[\log(T)]=\frac{\left(\psi(\phi)-\alpha\log\lambda+(\lambda+\phi)^{-1} \right)}{\alpha}  .
\end{equation}

\subsection{Survival Properties}

In this section, we present the survival, the hazard and mean residual life function for the GWL distribution. The survival function of $T\sim \f{GWL}(\phi,\lambda,\alpha)$ with the probability of an observation does not fail until a specified time $t$ is
\begin{equation}\label{fswl}
S(t|\phi,\lambda,\alpha) = \frac{\Gamma\left[\phi,(\lambda t)^{\alpha}\right](\lambda+\phi)+(\lambda t)^{\alpha\phi}e^{-(\lambda t)^\alpha}}{(\lambda+\phi)\Gamma(\phi)}\ 
\end{equation}
where $\Gamma(x,y)=\int_{0}^{x}{w^{y-1}e^{-x}dw}$ is called upper incomplete gamma. The hazard function quantify the instantaneous risk of failure at a given time  $t$ and is given by
\begin{equation}\label{fhwl} 
h(t|\phi,\lambda,\alpha)=\frac{f(t|\phi,\lambda,\alpha)}{S(t|\phi,\lambda,\alpha)}=\frac{\alpha\lambda^{\alpha\phi}t^{\alpha\phi-1}(\lambda+(\lambda t)^{\alpha})e^{-(\lambda t)^\alpha}}{\Gamma\left[\phi,(\lambda t)^{\alpha}\right](\lambda+\phi)+(\lambda t)^{\alpha\phi}e^{-(\lambda t)^\alpha}}. 
\end{equation}

The behaviours of the hazard function (\ref{fhwl}) when $t\rightarrow0$ and $t\rightarrow\infty$, respectively, are given by
\begin{equation*}
h(0)=
\begin{cases}
 \infty, & \text{if }\alpha\phi<1 \\
 \dfrac{\alpha\lambda^2}{(\lambda+\phi)\Gamma(\phi)}, & \text{if }\alpha\phi=1 \\
 0, & \text{if }\alpha\phi>1 
\end{cases}
\ \ \ \mbox{and}  \ \ \ h(\infty)=
\begin{cases}
 0, & \text{if }\alpha\phi<1 \\
 \lambda, & \text{if }\alpha\phi=1 \\
 \infty, & \text{if }\alpha\phi>1. 
\end{cases} \ \ 
\end{equation*}

\begin{theorem}The hazard rate function $h(t)$ of the generalized weighted Lindley distribution is increasing, decreasing, bathtub, unimodal or decreasing-increasing-decreasing shaped.
\end{theorem}
\begin{proof} Is not straightforward to apply the theorem proposed by Glaser (1980) in the GLW distribution. Moreover, since the hazard rate function  (\ref{fhwl}) is complex, we consider the following cases:

\begin{enumerate}
\item Let $\alpha=1$, then GWL distribution reduces to the WL distribution. In this case, Ghitany et al. (2008) proved that the hazard function is bathtub shaped (increasing) if $0<\phi<1$ $(\phi>0)$, for all $\lambda>0$.

\item Let $\phi=1$, then GWL distribution reduces to the PL distribution. In this case, considering $\beta=\lambda^\alpha$, Ghitany et al. (2013) proved that the hazard function is 
\begin{itemize}
\item increasing if $\left\{0<\alpha\geq 1,\beta>0 \right\}$;

\item decreasing if $\left\{0<\alpha\leq\frac{1}{2}, \beta>0 \right\}$ or $\left\{\frac{1}{2}<\alpha<1, \beta\geq {(2\alpha-1)}^2{(4\alpha(1-\alpha))}^{-1} \right\}$;

\item decreasing-increasing-decreasing if $\left\{\frac{1}{2}<\alpha<1, 0<\beta< {(2\alpha-1)}^2{(4\alpha(1-\alpha))}^{-1} \right\}$.

\end{itemize}
\vspace{-0.5cm}
\item  Let $\alpha=2$ and $\lambda=1$, from Glaser’s theorem (1980), we conclude straightforwardly that the hazard rate function is decreasing shaped (unimodal) if $0<\phi<1$ $(\phi>1)$.
\end{enumerate}
\vspace{-0.6cm}
\end{proof}

These properties make the GWL distribution a  flexible model for reliability data. Figure \ref{friswl} gives examples from the shapes of the hazard function for different values of $\phi, \lambda$ and $\alpha$.
\begin{figure}[!htb]
\centering
\includegraphics[scale=0.53]{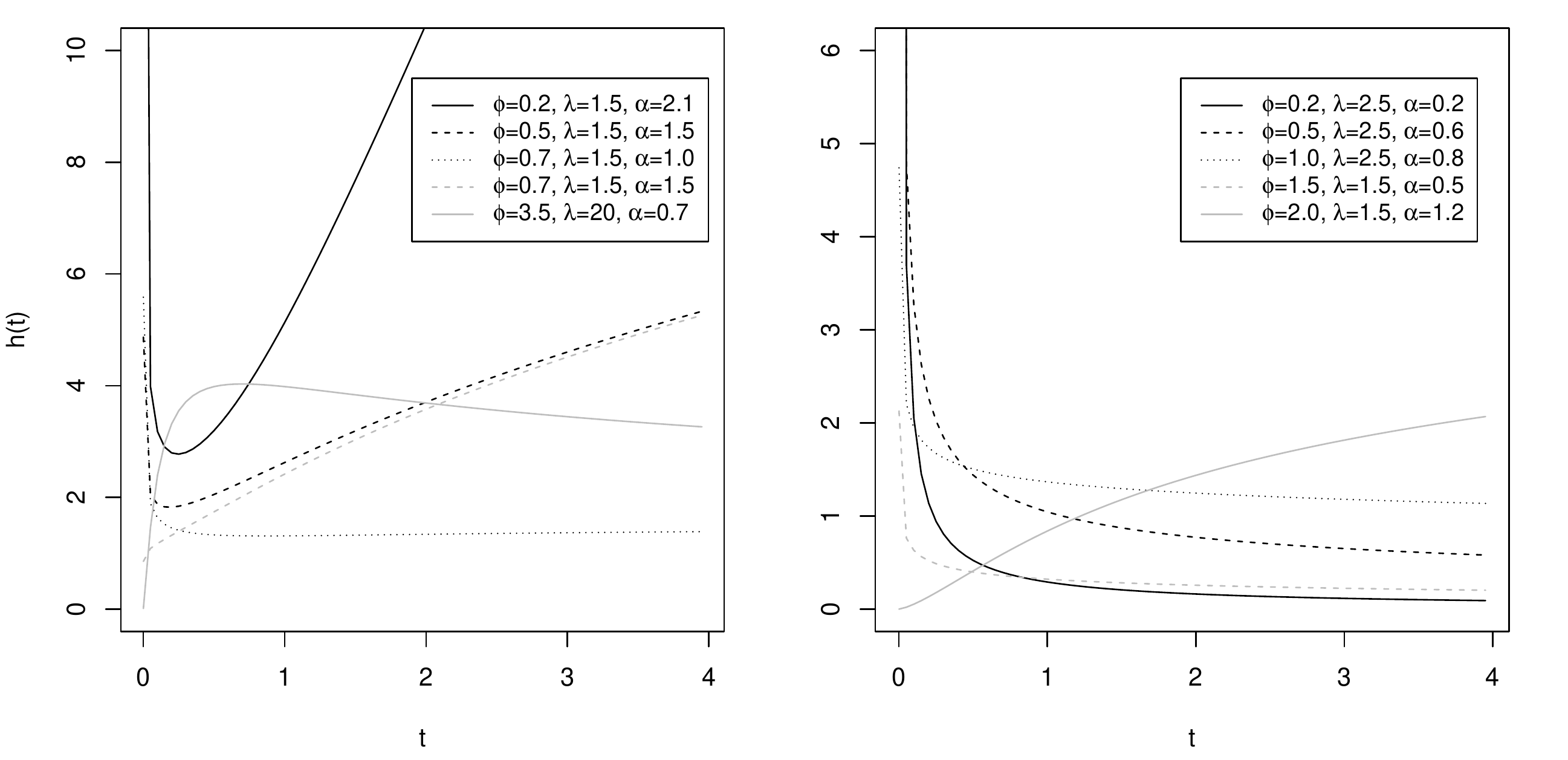}
\caption{Hazard function shapes for GWL distribution and considering different values of $\phi, \lambda$ and $\alpha$}\label{friswl}
\end{figure}

The mean residual life (MRL) has been used widely in survival analysis and represents the expected additional lifetime given that a component has survived until time t, the following result presents the MRL function of the $\f{GWL}$ distribution 

\begin{proposition} The mean residual life function $r(t|\phi,\lambda,\alpha)$ of the $\f{GWL}$ distribution is given by
\begin{equation}\label{mr1wl}
\begin{aligned} 
r(t|\phi,\lambda,\alpha)&=\frac{\left(\phi+\frac{1}{\alpha}+\lambda\right)\Gamma\left(\phi+\frac{1}{\alpha},\left(\lambda t\right)^\alpha\right)-\lambda t(\lambda+\phi)\Gamma\left(\phi,\left(\lambda t\right)^\alpha\right)}{\lambda[(\lambda+\phi)\Gamma(\phi,\left(\lambda t\right)^\alpha)+(\lambda t)^{\alpha\phi}e^{-\left(\lambda t\right)^\alpha}]}.
\end{aligned} 
\end{equation}
\end{proposition}
\begin{proof} Note that
\begin{equation*}
\begin{aligned} 
r(t|\phi,\lambda,\alpha)&= \frac{1}{S(t)}\int_{t}^{\infty}yf(y|\lambda,\phi)dy-t \\ &=\frac{1}{S(t)}\left[p\int_{t}^{\infty}yf_1(y|\lambda,\phi)dy+(1-p)\int_{x}^{\infty}yf_2(y|\lambda,\phi)dy\right]-t \\& =\frac{\left(\phi+\frac{1}{\alpha}+\lambda\right)\Gamma\left(\phi+\frac{1}{\alpha},\left(\lambda t\right)^\alpha\right)-\lambda t(\lambda+\phi)\Gamma\left(\phi,\left(\lambda t\right)^\alpha\right)}{\lambda[(\lambda+\phi)\Gamma(\phi,\left(\lambda t\right)^\alpha)+(\lambda t)^{\alpha\phi}e^{-\left(\lambda t\right)^\alpha}]}.
\end{aligned} 
\end{equation*}
\end{proof}
The behaviors of the mean residual life function when $t\rightarrow0$ and $t\rightarrow\infty$, respectively, are given by
\begin{equation*}
r(0)=\frac{1}{\lambda\left((\lambda+\phi)\Gamma(\phi)\right)} \ \ \ \mbox{and} \ \ \ r(\infty)
\begin{cases}
 \infty, & \text{if }\alpha<1 \\
 \dfrac{1}{\lambda}, & \text{if }\alpha=1 \\
 0, & \text{if }\alpha>1 
\end{cases}
.\ \ \ 
\end{equation*}



\subsection{Entropy}

In information theory, entropy has played a central role as a measure of the uncertainty associated with a random variable. Proposed by Shannon (1948), Shannon's entropy is one of the most important metrics in information theory. Shannon's entropy for the GWL distribution can be obtained by solving the following equation
\begin{equation}\label{she1} 
H_S(\phi,\lambda,\alpha)=-\int_{0}^{\infty}\log\left(\frac{\alpha\lambda^{\alpha\phi}t^{\alpha\phi-1}(\lambda+(\lambda t)^{\alpha})e^{-(\lambda t)^\alpha}}{(\lambda+\phi)\Gamma(\phi)}\right)f(t|\phi,\lambda,\alpha)dt.
\end{equation}	

\begin{proposition} A random variable $T$ with $\f{GWL}$ distribution,
has Shannon's Entropy  given by
\begin{equation}\label{shegwl} 
\begin{aligned}
H_S(\phi,\lambda,\alpha)=& \ \log(\lambda+\phi)+\log\Gamma(\phi)-\log\alpha-\log\lambda-\frac{\phi(1+\phi+\lambda)}{(\lambda+\phi)} \\ & -\frac{\psi(\phi)(\alpha\phi-1)}{\alpha}-\frac{(\alpha\phi-1)}{\alpha(\lambda+\phi)}-\frac{\eta(\phi,\lambda)}{(\lambda+\phi)\Gamma(\phi)}.
\end{aligned} 
\end{equation}
where
\begin{equation*}
\eta(\phi,\lambda)=\int_{0}^{\infty}(\lambda+y)\log(\lambda+y)y^{\phi-1}e^{-y}dy=\int_{0}^{1}(\lambda-\log u)\log(\lambda-\log u)(-\log u)^{\phi-1}du .
\end{equation*}

\begin{proof} From the equation (\ref{she1}) we have
\begin{equation}\label{shep1} 
\begin{aligned}
H_S(\phi,\lambda,\alpha)=& \ -\log\alpha-\alpha\phi\log\lambda+\log(\lambda+\phi)+\log(\Gamma(\phi))+\lambda^\alpha E[T^\alpha] \\&-(\alpha\phi-1)E[\log{T}] - E\left[\log(\lambda+(\lambda T)^{\alpha}) \right]
\end{aligned}
\end{equation}

Note that
\begin{equation*}\label{shep2}
E\left[\log(\lambda+(\lambda T)^{\alpha}) \right]=\int_{0}^{\infty}	\log(\lambda+(\lambda T)^{\alpha}\frac{\alpha\lambda^{\alpha\phi}t^{\alpha\phi-1}(\lambda+(\lambda t)^{\alpha})e^{-(\lambda t)^\alpha}}{(\lambda+\phi)\Gamma(\phi)}dt,
\end{equation*}
using the change of variable $y=(\lambda t)^{\alpha}$ and after some algebra 
\begin{equation*}\label{shep22}
\begin{aligned}
E\left[\log(\lambda+(\lambda T)^{\alpha}) \right]&=\frac{1}{(\lambda+\phi)\Gamma(\phi)}\int_{0}^{\infty}(\lambda+y)\log(\lambda+y)y^{\phi-1}e^{-y}dy \\&= \frac{\eta(\phi,\lambda)}{(\lambda+\phi)\Gamma(\phi)}.
\end{aligned}
\end{equation*}

Through equations (\ref{rmgwl}) and (\ref{lrmgwl}), we can easily find the solution of $E[T^\alpha]$ and $E[\log{T}]$ and the result follows.
\end{proof}
\end{proposition}

Other popular entropy measure is proposed by Renyi (1961). Some recent applications of the Reenyi entropy can be seen in  Popescu \& Aiordachioaie (2013). If $T$ has the probability density function (1) then Renyi entropy is defined by
\begin{equation}
\frac{1}{1-\rho}\log\int_{0}^{\infty}f^\rho(x)dx .
\end{equation}
\begin{proposition} A random variable $T$ with $\f{GWL}$ distribution,
has the Renyi entropy given by
\begin{equation}\label{renyigwl} 
H_R(\rho)= \frac{(\rho-1)(\log\alpha+\log\lambda)-\rho\left(\log(\lambda+\phi)+\log\Gamma(\phi)\right)-\log(\delta(\rho,\phi,\lambda,\alpha))}{1-\rho}
\end{equation}
where $\delta(\rho,\phi,\lambda,\alpha)=\int_{0}^{\infty}y^{\tfrac{\rho\phi-\rho+1-\alpha}{\alpha}}(\lambda+y)^\rho e^{-\rho y}dy$.
\begin{proof}
The Renyi entropy is given by
\begin{equation*}
\begin{aligned}
H_R(\rho)&=\frac{1}{1-\rho}\log\left(\frac{\alpha^\rho\lambda^\rho}{(\lambda+\phi)^\rho\Gamma(\phi)^\rho}\int_{0}^{\infty}(\lambda t)^{\alpha\rho\left(\phi-\tfrac{1}{\alpha}\right)}\left(\lambda+(\lambda t)^\alpha \right)^\rho e^{-\rho(\lambda t)^\alpha}dt \right) \\&=\frac{1}{1-\rho}\log\left(\frac{\alpha^\rho\lambda^\rho}{(\lambda+\phi)^\rho\Gamma(\phi)^\rho}\int_{0}^{\infty}y^{\tfrac{\rho\phi-\rho+1-\alpha}{\alpha}}(\lambda+y)^\rho e^{-\rho y}dy \right)\\&=\frac{1}{1-\rho}\log\left(\frac{\alpha^\rho\lambda^\rho}{(\lambda+\phi)^\rho\Gamma(\phi)^\rho}\delta(\rho,\phi,\lambda,\alpha) \right)
\end{aligned}
\end{equation*}
and with some algebra the proof is completed.
\end{proof}
\end{proposition}

\subsection{Lorenz curves}

The Lorenz curve (see Bonferroni, 1930) are well-known measures used in reliability, income inequality, life testing and renewal theory. The Lorenz curve for a non-negative T random variable is given through the consecutive plot of

\begin{equation*}
L\left(F(t)\right)=\frac{\int_{0}^{t}xf(x)dx}{\int_{0}^{\infty}xf(x)dx}=\frac{1}{\mu}\int_{0}^{t}xf(x)dx.
\end{equation*}

\begin{proposition}
The Lorenz curve of the GWL distribution is
\begin{equation*}
L\left(p\right)=\frac{\gamma\left(\phi+1+\frac{1}{\alpha},\left(\lambda F^{-1}(p)\right)^\alpha\right)+\lambda\gamma\left(\phi+\frac{1}{\alpha},\left(\lambda F^{-1}(p)\right)^\alpha\right)}{\left(\frac{1}{\alpha}+\phi+\lambda\right)\Gamma\left[\frac{1}{\alpha}+\phi\right]}
\end{equation*}
or
\begin{equation*}
L\left(p\right)=\frac{\left(\frac{1}{\alpha}+\phi+\lambda\right)\gamma\left(\phi+\frac{1}{\alpha},\left(\lambda F^{-1}(p)\right)^\alpha\right)-\left(\lambda F^{-1}(p)\right)^{\alpha\phi-1}e^{-\left(\lambda F^{-1}(p)\right)^\alpha}}{\left(\frac{1}{\alpha}+\phi+\lambda\right)\Gamma\left[\frac{1}{\alpha}+\phi\right]}
\end{equation*}
where $F^{-1}(p)=t_p$.
\end{proposition}

\section{Methods of estimation}

In this section we describe eight different estimation methods to obtain the estimates of the parameters $\phi, \lambda$ and $\alpha$ of the GWL distribution.

\subsection{Maximum Likelihood Estimation}

Among the statistical inference methods, the maximum likelihood method is widely used due its better asymptotic properties. Under the maximum likelihood method, the estimators are obtained from maximizing the likelihood function (see for example, Casella \& Berger, 2002). Let $T_1,\ldots,T_n$ be a random sample such that $T\sim \f {GWL}(\phi,\lambda,\alpha)$. In this case, the likelihood function from (\ref{fdpgwl}) is given by,

\begin{equation}\label{veroggwl}
L(\phi,\lambda,\alpha;\boldsymbol{t})=\frac{\alpha^n\lambda^{n\alpha\phi}}{(\lambda+\phi)\Gamma(\phi)^n}\left\{\prod_{i=1}^n{t_i^{\alpha\phi-1}}\right\}\prod_{i=1}^n{\left(\lambda+ (\lambda t_i)^\alpha \right)}\exp\left\{-\lambda^{\alpha}\sum_{i=1}^n t_i^\alpha\right\}. \end{equation}

The log-likelihood function $l(\phi,\lambda,\alpha;\boldsymbol{t})=\log{L(\phi,\lambda,\alpha;\boldsymbol{t})}$ is given by,
\begin{equation}\label{veroggwl2}
\begin{aligned}
l(\phi,\lambda,\alpha;\boldsymbol{t})=\ & n\log\alpha + n\alpha\phi\log\lambda-n\log(\lambda+\phi)-n\log\Gamma(\phi) +(\alpha\phi-1)\sum_{i=1}^{n}\log(t_i) \\ &+ \sum_{i=1}^{n}\log\left(\lambda+ (\lambda t_i)^\alpha \right)-\lambda^{\alpha}\sum_{i=1}^n t_i^\alpha .
\end{aligned}
\end{equation}

From the expressions $\frac{\partial}{\partial \phi}l(\phi,\lambda,\alpha;\boldsymbol{t})=0$, $\frac{\partial}{\partial \lambda}l(\phi,\lambda,\alpha;\boldsymbol{t})=0$, $\frac{\partial}{\partial \alpha}l(\phi,\lambda,\alpha;\boldsymbol{t})=0$, we get the likelihood equations,
\begin{equation}\label{verogg21} 
n\hat\alpha\log(\hat\lambda)+\hat\alpha\sum_{i=1}^{n}\log(t_i)=\frac{n}{\hat\lambda+\hat\phi} +n\psi(\hat\phi)
\end{equation}
\begin{equation}\label{verogg22} 
 \frac{n\hat\alpha\hat\phi}{\hat\lambda}+\sum_{i=1}^{n}\frac{1+\hat\alpha\hat\lambda^{\hat\alpha-1}t_{i}^{\hat\alpha}}{\hat\lambda+(t_i)^{\hat\alpha}}= \hat{\alpha}\hat{\lambda}^{\hat\alpha-1}\sum_{i=1}^n t_{i}^{\hat\alpha}+\frac{n}{\hat\lambda+\hat\phi} 
\end{equation}
\begin{equation}\label{verogg23} 
\dfrac{n}{\hat{\alpha}}+n\hat\phi\log(\hat\lambda)+\hat\phi\sum_{i=1}^n\log(t_i)+\sum_{i=1}^n\frac{ (\hat\lambda t_{i})^{\hat\alpha}\log(\hat\lambda t_i)}{\hat\lambda+(\hat\hat\lambda t_{i})^{\hat\alpha}}=\hat\lambda^{\hat\alpha}\sum_{i=1}^n{t_i}^{\hat\alpha}\log(\hat\lambda t_i), 
\end{equation}
where $\psi(k)=\frac{\partial}{\partial k}\log\Gamma(k)=\frac{\Gamma'(k)}{\Gamma(k)}$. The solutions of  such non-linear system provide the maximum likelihood estimators. Numerical methods such as Newton-Rapshon are required to find the solution of the nonlinear system. Note that from (\ref{verogg21}) and (\ref{verogg23}) and after some algebra we have
\begin{equation}\label{verogg24} 
\hat\alpha=\frac{1}{\left(n\log(\hat\lambda)+\sum_{i=1}^{n}\log(t_i)\right)}\left(\frac{n}{\hat\lambda+\hat\phi} +n\psi(\hat\phi)\right)
\end{equation}
\begin{equation}\label{verogg25} 
\hat\phi=\frac{\left(\hat\lambda^{\hat\alpha}\sum_{i=1}^n{t_i}^{\hat\alpha}\log(\hat\lambda t_i)-\sum_{i=1}^n\frac{ (\hat\lambda t_{i})^{\hat\alpha}\log(\hat\lambda t_i)}{\hat\lambda+(\hat\hat\lambda t_{i})^{\hat\alpha}}-\dfrac{n}{\hat{\alpha}}\right)}{\left(n\log(\hat\lambda)+\sum_{i=1}^{n}\log(t_i)\right)}, 
\end{equation}

The obtained MLE's (maximum likelihood estimators) of $\phi, \lambda$ and $\alpha$ are biased considering small sample sizes. For large sample sizes the obtained estimators are not biased and they are asymptotically efficient. The MLE estimates are asymptotically normal distributed with a joint multivariate normal distribution given by,
\begin{equation} (\hat\phi,\hat\lambda,\hat\alpha) \sim N_3[(\phi,\lambda,\alpha),I^{-1}(\phi,\lambda,\alpha)] \mbox{ for } n \to \infty , \end{equation}
where $I(\phi,\lambda,\alpha)$ is the Fisher information matrix given by,
\begin{equation}\label{mfishergg}
I(\phi,\lambda,\alpha)=
\begin{bmatrix}
I_{\phi,\phi}(\phi,\lambda,\alpha) & I_{\phi,\lambda}(\phi,\lambda,\alpha) & I_{\phi,\alpha}(\phi,\lambda,\alpha) \\
I_{\phi,\lambda}(\phi,\lambda,\alpha) & I_{\lambda,\lambda}(\phi,\lambda,\alpha)  & I_{\lambda,\alpha}(\phi,\lambda,\alpha) \\
I_{\phi,\alpha}(\phi,\lambda,\alpha) & I_{\lambda,\alpha}(\phi,\lambda,\alpha) & I_{\alpha,\alpha}(\phi,\lambda,\alpha)
\end{bmatrix} ,
\end{equation}
where the elements of the matrix 𝐼are given in the appendix. 

\subsection{Moments Estimators}
The method of moments is one of the oldest method used for estimating parameters in statistical models. The moments estimators (MEs) of the GLW distribution can be obtained by equating the first three theoretical moments (\ref{rmgwl}) with the sample moments $\bar{x}=\frac{1}{n}\sum_{i=1}^n t_i$, $\frac{1}{n}\sum_{i=1}^n t_i^2$ and $\frac{1}{n}\sum_{i=1}^n t_i^3$ respectively,
\begin{equation*}\label{moments1}
\frac{1}{n}\sum_{i=1}^n t_i = \frac{\left(\frac{1}{\alpha}+\phi+\lambda\right)\Gamma(\frac{1}{\alpha}+\phi)}{(\lambda+\phi)\lambda \Gamma(\phi)} \ ,  \ \frac{1}{n}\sum_{i=1}^n t_i^2 = \frac{\left(\frac{2}{\alpha}+\phi+\lambda\right)\Gamma(\frac{2}{\alpha}+\phi)}{(\lambda+\phi)\lambda^2 \Gamma(\phi)}
\end{equation*}
\begin{equation*}
\mbox{ and } \ \frac{1}{n}\sum_{i=1}^n t_i^3 = \frac{\left(\frac{3}{\alpha}+\phi+\lambda\right)\Gamma(\frac{3}{\alpha}+\phi)}{(\lambda+\phi)\lambda^3 \Gamma(\phi)}.
\end{equation*}

Therefore, the moments estimators $\hat{\phi}_{ME}$, $\hat{\lambda}_{ME}$ and $\hat{\alpha}_{ME}$, can be obtained by solving the non-linear equation
\begin{equation*}
 \frac{\left(\frac{j}{\alpha}+\phi+\lambda\right)\Gamma(\frac{j}{\alpha}+\phi)}{(\lambda+\phi)\lambda^j \Gamma(\phi)}-\frac{1}{n}\sum_{i=1}^n t_i^j =0 , \quad  j=1,2,3.
\end{equation*}

\subsection{Ordinary and Weighted Least-Square Estimate}
Let $t_{(1)}, t_{(2)},\cdots,t_{(n)}$ denotes the order statistics (we assume the same notation for the next subsections) of the  random sample of size $n$ from a distribution function $F(\boldsymbol{t}| \phi,\lambda,\alpha)$. The least square estimators $\hat{\phi}_{LSE}$, $\hat{\lambda}_{LSE}$ and $\hat{\alpha}_{LSE}$, can be obtained by minimizing
\begin{equation*}
V\left( \phi,\lambda,\alpha\right) = \sum_{i=1}^{n}\left[ F\left( t_{(i)}\mid \phi,\lambda,\alpha \right) - \frac {i}{n+1} \right]^{2},
\end{equation*}
with respect to $\phi, \lambda $ and $\alpha$, where $F(\boldsymbol{t}| \phi,\lambda,\alpha)$ is given by (\ref{densagwl}).
Equivalently, they can be obtained by solving the non-linear equations:
\begin{equation*}
\begin{aligned}
&\sum_{i=1}^{n}\left[ F\left( t_{(i)}\mid \phi,\lambda,\alpha \right) - \frac {i}{n+1}\right] \Delta_{j}\left( t_{(i)} \mid \phi,\lambda,\alpha \right) = 0, \quad  j=1,2,3.
\end{aligned}
\end{equation*}
where
\begin{equation}\label{delta1}
\begin{aligned}
\Delta_{1}\left( t_{(i)}\mid \phi,\lambda,\alpha \right) = &\frac{\partial}{\partial \phi} F\left( t_{(i)}\mid \phi,\lambda,\alpha \right), \,
\Delta_{2}\left( t_{(i)}\mid \phi,\lambda,\alpha \right) = \frac{\partial}{\partial \lambda} F\left( t_{(i)}\mid \phi,\lambda,\alpha \right) \\ &
\mbox{ and } \Delta_{3}\left( t_{(i)}\mid \phi,\lambda,\alpha \right) = \frac{\partial}{\partial \alpha} F\left( t_{(i)}\mid \phi,\lambda,\alpha \right)
\end{aligned}
\end{equation}%

Note that, the computation of $\Delta_{i}$ for $i=1,2,3$ involves the solutions of partial derivatives of the lower incomplete gamma function. However this can be easily done numerically with high precision.

The weighted least-squares estimates (WLSEs), $\hat{\phi}_{WLSE}$, $\hat{\lambda}_{WLSE}$ and $\hat{\alpha}_{WLSE}$, can be obtained by minimizing
\begin{equation*}
W\left( \phi,\lambda,\alpha \right) = \sum_{i=1}^{n}
\frac {\left( n+1\right)^{2}\left( n+2\right)}{i\left( n-i+1\right)}
\left[ F\left( t_{(i)}\mid \phi,\lambda,\alpha \right) - \frac {i}{n+1} \right]^{2}.
\end{equation*}
These estimates can also be obtained by solving the non-linear equations:
\begin{equation*}
\sum_{i=1}^{n}\frac {\left( n+1\right)^{2}\left( n+2\right)}{i\left( n-i+1\right)}
\left[ F\left( t_{(i)}\mid \phi,\lambda,\alpha \right) -
\frac {i}{n+1} \right] \Delta_{j}\left( t_{(i)}\mid \phi,\lambda,\alpha\right) = 0, \quad j=1,2,3,
\end{equation*}
where $\Delta _{1}\left( \cdot \mid \phi,\lambda,\alpha \right)$, $\Delta _{2}\left( \cdot \mid \phi,\lambda,\alpha \right) $ and
$\Delta
_{3}\left( \cdot \mid \phi,\lambda,\alpha \right) $ are given respectively in (\ref{delta1}).

\subsection{Method of Maximum Product of Spacings}

The maximum product of spacings (MPS) method is a powerful alternative to MLE for the estimation of the unknown parameters of continuous univariate distributions. Proposed by Cheng and Amin (1979, 1983), these method was also independently developed by Ranneby (1984) as an approximation to the Kullback-Leibler information measure. Cheng and Amin (1983) proved desirable properties of the MPS such as, asymptotic efficiency and invariance, they also proved that the consistency of maximum product of spacing estimators holds under much more general conditions than for maximum likelihood estimators.

 Let $D_{i}(\phi,\lambda,\alpha)=F\left( t_{(i)}\mid\phi,\lambda,\alpha \right)
-F\left( t_{(i-1)}\mid \phi,\lambda,\alpha \right)$, for  $i=1,2,\ldots
,n+1,$ be the uniform spacings of a random sample from the GWL distribution, where $F(t_{(0)}\mid \phi,\lambda,\alpha)=0$ and $F( t_{(n+1)}\mid
\phi,\lambda,\alpha)=1.$ Clearly \
$\sum_{i=1}^{n+1} D_i (\phi,\lambda,\alpha) =1$. The maximum product of spacings estimates $\hat{\phi}_{MPS}$, $\hat{\lambda}_{MPS}$ and $\hat{\alpha}_{MPS}$ are obtained by maximizing the geometric
mean of the spacings
\begin{equation}
G\left( \phi,\lambda,\alpha\right) =\left[ \prod\limits_{i=1}^{n+1}D_{i}( \phi,\lambda,\alpha)\right] ^{%
\frac{1}{n+1}},  \label{G}
\end{equation}%
with respect to $\phi$, $\lambda$ and $\alpha$, or, equivalently, by maximizing  the logarithm of the geometric mean of sample spacings
\begin{equation}
H\left( \phi,\lambda,\alpha\right) =\frac{1}{n+1}\sum_{i=1}^{n+1}\log
D_{i} ( \phi,\lambda,\alpha).
\end{equation}

The estimates $\hat{\phi}_{MPS}$, $\hat{\lambda}_{MPS}$ and $\hat{\alpha}_{MPS}$ 
of the parameters $\phi$, $\lambda$ and $\alpha$ can  be obtained by solving
the nonlinear equations%
\begin{equation}
\frac{1}{n+1}%
\sum\limits_{i=1}^{n+1}\frac{1}{D_{i}(\phi,\lambda,\alpha))} \left[ \Delta_j
(t_{(i)} |  \phi,\lambda,\alpha) - \Delta_j (t_{(i-1)} |  \phi,\lambda,\alpha)
\right] =0, \quad j=1,2,3,
\end{equation}
where $\Delta _{1}\left( \cdot \mid \phi,\lambda,\alpha \right)$, $\Delta _{2}\left( \cdot \mid \phi,\lambda,\alpha \right) $ and
$\Delta
_{3}\left( \cdot \mid \phi,\lambda,\alpha \right) $ are given respectively in (\ref{delta1}). Note that if $t_{(i+k)}=t_{(i+k-1)}=\ldots=t_{(i)}$ we get $D_{i+k}(\phi,\lambda,\alpha)=D_{i+k-1}(\phi,\lambda,\alpha)=\ldots=D_{i}(\phi,\lambda,\alpha)=0$. Therefore, the MPS estimators are sensitive to closely spaced observations, especially ties. When the ties are due to multiple observations, $D_{i}(\phi,\lambda,\alpha)$ should be replaced by the corresponding  likelihood $f(t_{(i)},\phi,\lambda,\alpha)$ since $t_{(i)}=t_{(i-1)}$. Under mild conditions for the GWL distribution the MPS estimators are asymptotically normal distributed (see Cheng and Stephens, 1989, for more details) with a joint trivariate normal distribution given by,
\begin{equation*}
(\hat{\phi}_{MPS},\hat{\lambda}_{MPS},\hat{\alpha}_{MPS})\sim N_3\left[(\phi,\lambda,\alpha),I^{-1}(\phi,\lambda,\alpha))\right] \mbox{ as } n \to \infty .
\end{equation*}

\subsection{The Cramer-von Mises minimum distance estimators}

The Cramer-von Mises estimator is a type of minimum distance estimators (also called maximum goodness-of-fit estimators) and is based on the difference between the estimate of the cumulative distribution function and the empirical distribution
function (see, D'Agostino \& Stephens, 1986; Luce\~{n}o, 2006).

MacDonald (1971) motivate the choice of the Cramer-von Mises type minimum distance estimators providing empirical evidence that the bias of the estimator is smaller than the other minimum distance estimators. The Cramer-von Mises estimates $\hat{\phi}_{CME}$, $\hat{\lambda}_{CME}$ and $\hat{\alpha}_{CME}$ of the parameters $\phi$, $\lambda$ and $\alpha$ are
obtained by minimizing 
\begin{equation}
C(\phi,\lambda,\alpha) =\frac{1}{12n}+\sum_{i=1}^{n}\left( F\left(
t_{(i)}\mid \phi,\lambda,\alpha\right) -{\frac{2i-1}{2n}}\right) ^{2},
\end{equation}
with respect to $\phi$, $\lambda$ and $\alpha$. These estimates can also be obtained by solving the non-linear equations:
\begin{equation*}
\sum_{i=1}^{n}\left( F\left( t_{(i)}\mid \phi,\lambda,\alpha \right) -{\frac{2i-1}{2n%
}}\right) \Delta _{j}\left( t_{(i)}\mid \phi,\lambda,\alpha \right)  =0, \quad j=1,2,3,
\end{equation*}%
where $\Delta _{1}\left( \cdot \mid \phi,\lambda,\alpha \right)$, $\Delta _{2}\left( \cdot \mid \phi,\lambda,\alpha \right) $ and
$\Delta
_{3}\left( \cdot \mid \phi,\lambda,\alpha \right) $ are given respectively in (\ref{delta1}).

\subsection{The Anderson-Darling and Right-tail Anderson-Darling estimators}
Other type of minimum distance estimators  is based on an Anderson-Darling statistic and is known as the Anderson-Darling estimator. The Anderson-Darling estimates $\widehat{\phi }_{ADE}, \widehat{\lambda }_{ADE}$ and $\widehat{\alpha}%
_{ADE}$ of the parameters $\phi, \lambda$ and $\alpha$ are obtained
by minimizing, with respect to $\phi$, $\lambda$ and $\alpha$, the function
\begin{equation}
A(\phi,\lambda,\alpha) =-n-\frac{1}{n}\sum_{i=1}^{n}\left( 2i-1\right)
\left(\, \log F\left( t_{(i)}\mid \phi,\lambda,\alpha \right)+ \log S\left(
t_{(n+1-i)}\mid \phi,\lambda,\alpha\right) \, \right) .
\end{equation}
These estimates can also be obtained by solving the non-linear equations:
\begin{equation*}
\sum_{i=1}^{n}\left( 2i-1\right) \left[ \frac{\Delta _{j}\left( t_{(i)}\mid
\phi,\lambda,\alpha \right) }{F\left( t_{(i)}\mid \phi,\lambda,\alpha \right) }-
\frac{\Delta_{j}\left( t_{(n+1-i)}\mid \phi,\lambda,\alpha \right) }{S\left( t_{(n+1-i)}\mid \phi,\lambda,\alpha \right) }\right] =0, \ \ j=1,2,3.
\end{equation*}

The Right-tail Anderson-Darling estimates  $\widehat{\phi }_{RADE}, \widehat{\lambda }_{RADE}$ and $\widehat{\alpha}%
_{RADE}$ of the parameters $\phi, \lambda$ and $\alpha$ 
are obtained by minimizing, with respect to $\phi $, $\lambda$ and $\alpha$, the function:%
\begin{equation}
R(\phi,\lambda,\alpha) =\frac{n}{2}-2\sum_{i=1}^{n}F\left( t_{i:n}\mid \phi,\lambda,\alpha \right) -\frac{1}{n}\sum_{i=1}^{n}\left( 2i-1\right) \log
S\left( t_{n+1-i:n}\mid \phi,\lambda,\alpha \right).
\end{equation}
These estimates can also be obtained by solving the non-linear
equations:
\begin{equation*}
- 2 \sum_{i=1}^{n} \Delta_{j}\left( t_{i:n}\mid \phi,\lambda,\alpha \right) +\frac{1}{n} \sum_{i=1}^{n}\left( 2i-1\right) \frac{\Delta_{j}\left( t_{_{n+1-i:n}}\mid \phi,\lambda,\alpha \right) }{S%
\left( t_{n+1-i:n}\mid \phi,\lambda,\alpha \right) } =0, \ \ j=1,2,3. 
\end{equation*}%
where $\Delta _{1}\left( \cdot \mid \phi,\lambda,\alpha \right)$, $\Delta _{2}\left( \cdot \mid \phi,\lambda,\alpha \right) $ and
$\Delta
_{3}\left( \cdot \mid \phi,\lambda,\alpha \right) $ are given respectively in (\ref{delta1}).\\

\section{Simulation Study}

In this section we develop an intensive simulation study to compare the efficiency of the different estimation procedures proposed for parameters of the GWL distribution. The following procedure was adopted:

\begin{enumerate}

\item Generate pseudo-random values from the $\f {GWL}(\phi,\lambda,\alpha)$ with size $n$.

\item Using the values obtained in step 1, calculate $\hat{\phi}$, $\hat{\lambda}$ and $\hat{\alpha}$ via 1-MLE, 2-MPS, 3-ADE, 4-RTADE, 5-LSE, 6-WLSE, 7-ME, 8-CME.

\item Repeat the steps $1$ and $2$ $N$ times.

\item Using $\boldsymbol{\hat{\theta}}=(\hat{\phi},\hat{\lambda},\hat{\alpha})$ and $\boldsymbol{\theta}=(\phi,\lambda,\alpha)$, compute
the mean relative estimates (MRE) $\sum_{j=1}^{N}\dfrac{\hat{\theta} 
_{i,j}/\theta_i}{N}$ and the mean square errors (MSE) $\sum_{j=1}^{N}\frac{(\hat{\theta}_{i,j}-\theta _{i})^{2}}{N}$, for $i=1,2,3$.
\end{enumerate}

Considering this approach it is expected that the most efficient estimation method will have MRE's closer to one with MSE's closer to zero. The results were computed using the software R (R Core Team, 2014) using the seed 2015 to generate the pseudo-random values. The chosen values to perform this procedure were $N=10000$ and $n=(50,60,\ldots,250)$. It will be present here results only for $\boldsymbol{\theta}=(2,0.5,0.1)$ for reasons of space. Nevertheless the following results were similar for other choices of $\boldsymbol{\theta}$. Moreover, for this comparison being meaningful, the estimation procedures need to be performed under the same conditions. However, for some particular samples and estimation methods the numerical techniques does not work well in finding the parameters estimates. Therefore, in Figure \ref{fsimulation4} it will be firstly presented the proportion of failure from each method.

\begin{figure}[!htb]
\centering
\includegraphics[scale=0.56]{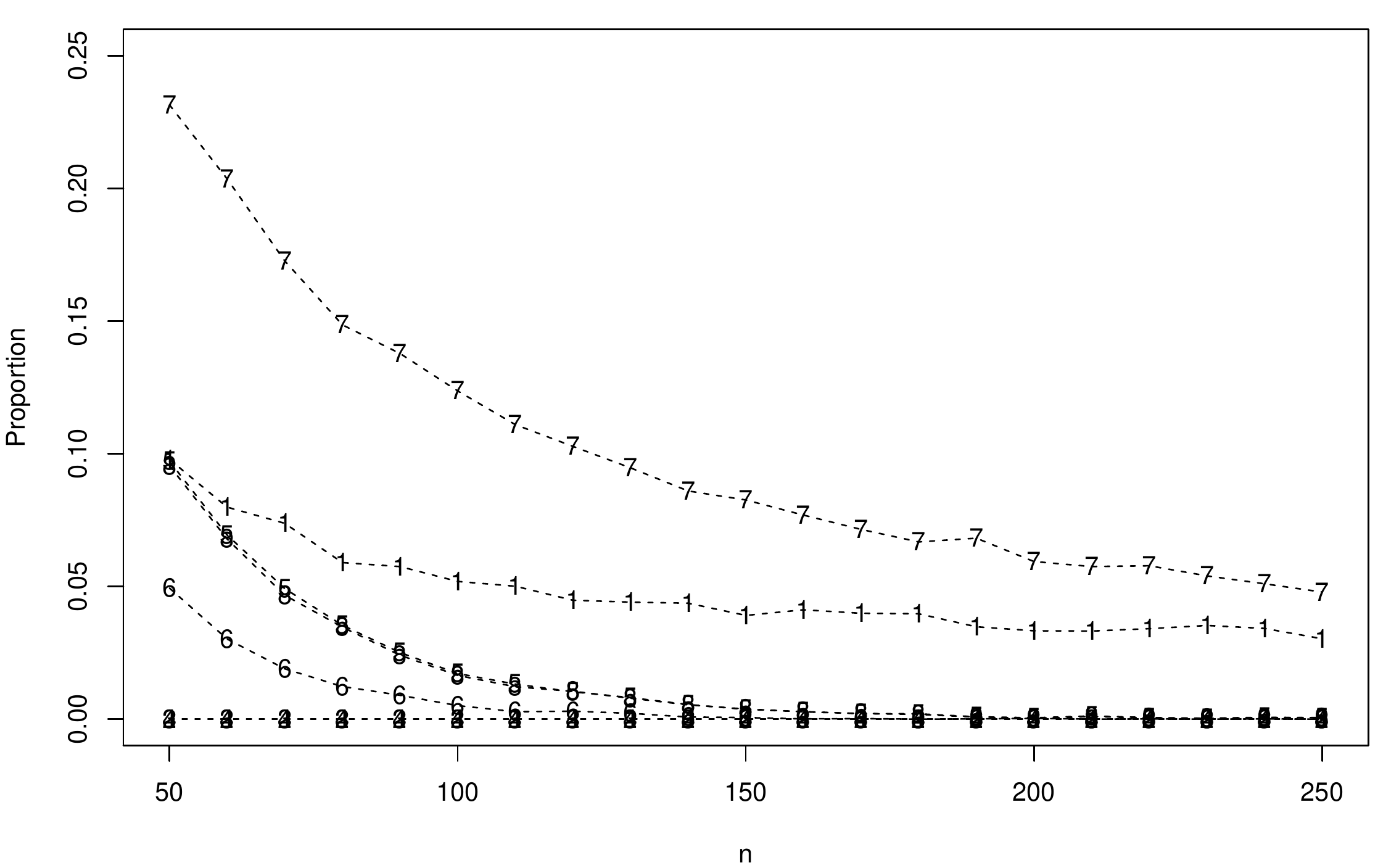}
\caption{Proportion of failure from $N$ simulated samples, considering different values of $n$ obtained using the following estimation method 1-MLE, 2-MPS, 3-ADE, 4-RTADE, 5-LSE, 6-WLSE, 7-ME, 8-CME.}\label{fsimulation4}
\end{figure}

\begin{figure}[!htb]
\centering
\includegraphics[scale=0.72]{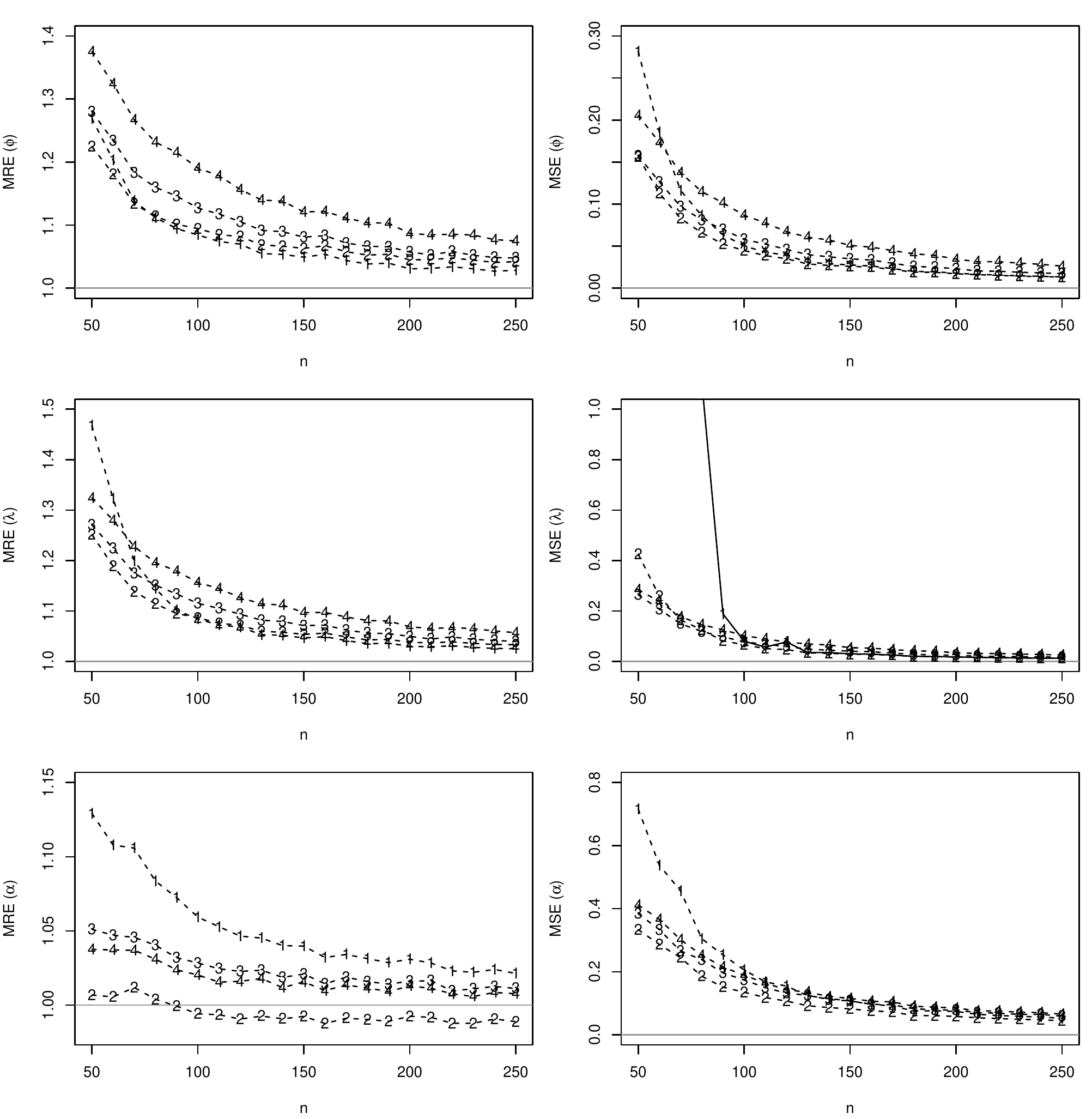}
\caption{MRE's, MSE's related from the estimates of $\phi=0.5, \lambda=0.7$ and $\alpha=1.5$  for N simulated samples, considering different values of $%
n $ obtained using the following estimation method 1-MLE, 2-MPS, 3-ADE, 4-RTADE.}
\label{fsimulation1}
\end{figure}

From Figure \ref{fsimulation4}, we note that the MLE, LSE, WLSE, ME and the CME estimators fail in finding the parameters estimates for a significant number of samples. Therefore, the use of such methods are not recommended for the GLW and we discard such estimation procedures. From now on we consider the MPS, ADE, RADE estimators and also the MLE only for illustrative purpose since it is the most widely used estimation method. Figures \ref{fsimulation1} presents MRE's, MSE's from the estimates of $\phi, \lambda$ and $\alpha$ obtained using the MLE, MPS, ADE, RADE for N simulated samples and considering different values of $\boldsymbol{\theta}=(2,0.5,0.1)$ and $n$. The horizontal lines in both figures corresponds to MRE's and MSE's being respectively one and zero. 

Based on these results, we observe that the MSE of the MLE, MPS, ADE and RADE estimators tend to zero for large $n$ and also, as expected, the values of
MRE's tend to one, i.e. the estimates are consistent and asymptotically unbiased for the parameters. For small sample sizes the MLE has the largest MSE's. The MPS has the smaller MSE's with MRE's closer to one for almost all values of n. Additionally, the MPS, ADE and RADE estimators were the only ones that were able to find $\hat\phi,\, \hat\lambda$ and $\hat\alpha$ for all the $2\times 10^6$ generated samples. Therefore, combining all results with the good properties of the MPS method such as consistency, asymptotic efficiency, normality and invariance we conclude that the MPS estimators is a
highly competitive method compared to the maximum likelihood for estimating
the parameters of the  GWL distribution.

\vspace{0.3cm}
\section{Application}
\vspace{0.3cm}

In this section, we compare the GWL distribution fit with several usual three parameters lifetime distributions considering two data sets one with bathtub hazard rate and one with the increasing hazard function. For sake of comparison the following lifetime distributions were considered: The generalized gamma (GG) distribution (Stacy, 1962) with p.d.f
given by  $f(t)= \alpha\,\Gamma(\phi)^{-1}\,\beta^{\alpha\phi}\,t^{\alpha\phi-1}\,e^{-(\beta t)^{\alpha}}$, where $\beta>0, \phi>0$ and $\alpha>0$, the generalized Weibull (GW) distribution (Mudholkar et al., 1996) with p.d.f given by  $f(t)= {(\alpha\phi)}^{-1}{(t/\phi)}^{1/\alpha-1}{(1-\lambda{(t/\phi)}^{1/\alpha} )}^{1/\lambda-1}$, where $\lambda\in\mathbb{R}$ the generalized exponential-Poisson (GEP) distribution (Barreto-Souza and Cribatari-Neto, 2009) with p.d.f given by  $f(t)= \left(\alpha\beta\phi/{(1-e^{-\phi})}^{\alpha}\right)e^{-\phi-\beta t+\phi\exp(-\beta t)}{\left(1-e^{-\phi+\phi\exp(-\beta t)}\right)}^{\alpha-1}\quad$ and the exponentiated Weibull (EW) distribution (Mudholkar et al., 1995) with p.d.f  $f(t)= \alpha\phi{\beta}^{-1}{(t/\beta)}^{\alpha-1}\times$ $\times\exp\left(-{(t/\beta)}^{\alpha} \right) \left(1-\exp\left(-{(t/\beta)}^{\alpha} \right) \right)^{\phi-1}
$.

Firstly,  it will be considered the TTT-plot (total time on test) in order to verify the behaviour of the empirical hazard function. Developed by Barlow and Campo (1975) the TTT-plot is achieve through plot of the values $[r/n,G(r/n)]$ where $G(r/n)= \left(\sum_{i=1}^{r}t_i +(n-r)t_{(r)}\right)$ $/{\sum_{i=1}^{n}t_i}$, $r=1,\ldots,n, i=1,\ldots,n$ and $t_{(i)}$ is the order statistics. If the curve is concave (convex), the hazard function is increasing (decreasing). When it starts convex and then concave (concave and then convex) the hazard function will have a bathtub (inverse bathtub) shape. Secondly, the discrimination criterion methods are: AIC (Akaike Information Criteria) and AICc (Corrected Akaike information criterion) computed respectively by $AIC=-2l(\hat{\boldsymbol{\theta}};\boldsymbol{t})+2k$ and $AICc=AIC+2\,k\,(k+1){(n-k-1)}^{-1}$, where $k$ is the number of parameters to be fitted and $\hat{\boldsymbol{\theta}}$ is estimation of $\boldsymbol{\theta}$. Given a set of candidate models for $\boldsymbol{t}$, the best one provide the minimum values. To check the goodness of fit it will be consider the Kolmogorov-Smirnov (KS) test . This procedure is based on the KS statistic $D_n=\sup\left\vert F_n(t)-F(t;\phi ,\lambda ,\alpha)\right\vert$, where $\sup t$ is the supremum of the set of distances, $F_n(t)$ is the empirical distribution function and $F(t;\alpha ,\beta ,\lambda)$ is c.d.f. In these case, testing the null hypothesis that the data comes from $F(t;\alpha ,\beta ,\lambda)$, and with significance level of $5\%$, we will reject the null hypothesis if the returned p-value is smaller than $0.05$.

\subsection{Lifetimes data}

Presented by Aarset (1987) in table \ref{arsetdd} is available the dataset is related to the lifetime in hours of $50$ devices put on test

\begin{table}[ht]
\caption{Lifetimes data (in hours) related to a device put on test.}
\centering 
\begin{tabular}{c c c c c c c c c c c c c} 
\hline 
0.1 & 0.2 & 1 & 1 & 1 & 1 & 1 & 2 & 3 & 6 & 7 & 11 & 12 \\
18 & 18 & 18 & 18 & 18 & 21 & 32 & 36 & 40 & 45 & 46 & 47 & 50 \\
55 & 60 & 63 & 63 & 67 & 67 & 67 & 67 & 72 & 15 & 79 & 82 & 82 \\
83 & 84 & 84 & 84 & 85 & 85 & 85 & 85 & 85 & 86 & 86 \\ [0ex] 
\hline 
\end{tabular}\label{arsetdd}
\end{table}

Figure \ref{grafico-obscajust2} shows (left panel) the TTT-plot, (middle
panel) the fitted survival superimposed to the empirical survival function and (right panels) the hazard function adjusted by GWL distribution. Table 9 presents  the AIC and AICc criteria and the p-value from the KS test for all fitted distributions considering the Aarset dataset.

\begin{figure}[!htb]
\centering
\includegraphics[scale=0.55]{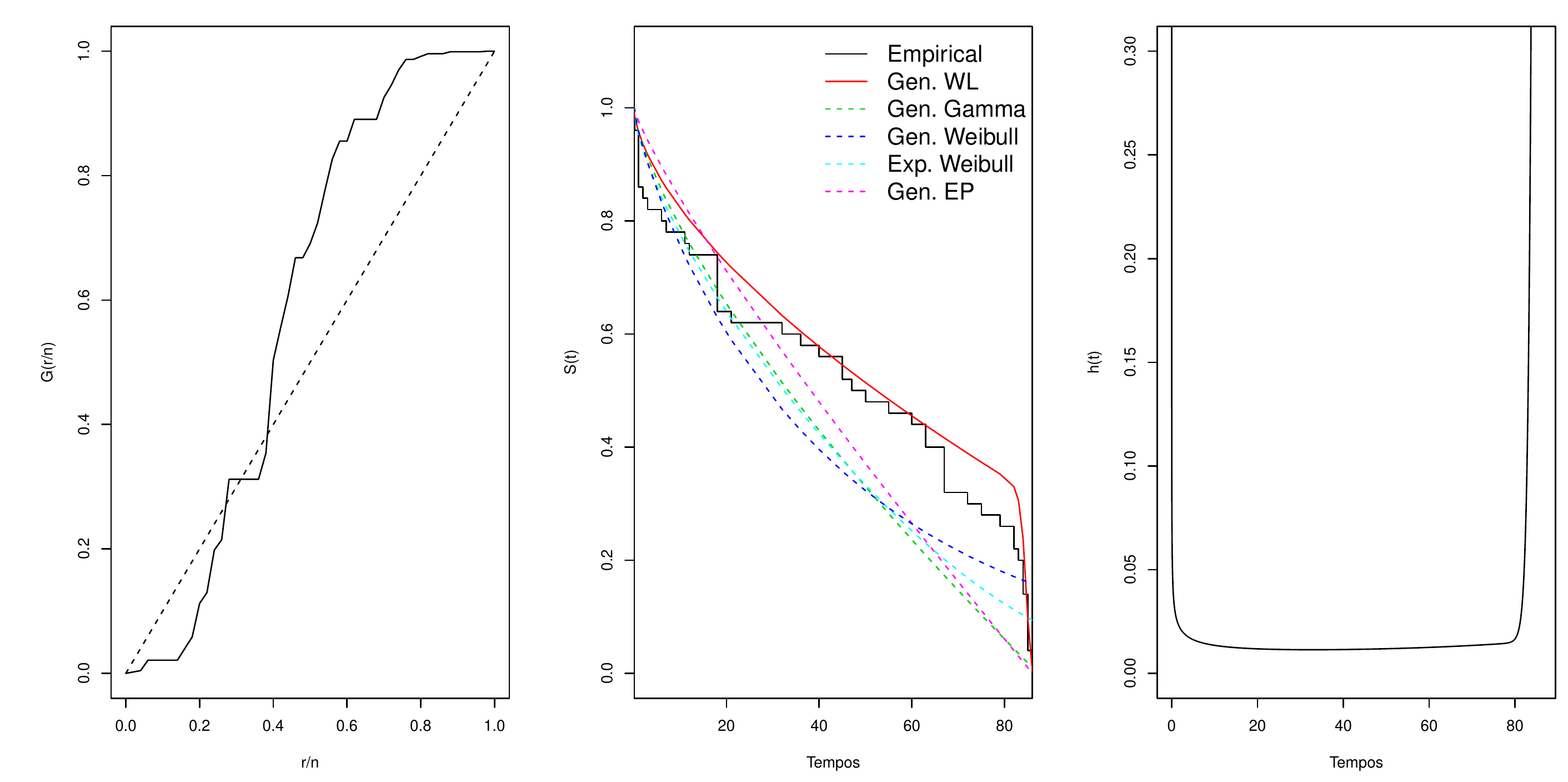}
\caption{(left panel) the TTT-plot, (middle
panel) the fitted survival superimposed to the empirical survival function and (right panels) the hazard function adjusted by GWL distribution}\label{grafico-obscajust2}
\end{figure}

\begin{table}[ht]
\caption{Results of AIC and AICc criteria and the p-value from the KS test for all fitted distributions considering the Aarset dataset.}
\centering 
\begin{center}
  \begin{tabular}{ c | c | c | c | c | c}
    \hline
		Criteria & Gen. WL & \ Gen. Gamma  \ & \ Gen. Weibull &  \ Exp. Weibull &  \   Gen. EP \\ \hline
      AIC & \textbf{418.031} & 448.294 & 430.055 & 463.674 & 486.255   \\ \hline
			AICc & \textbf{412.552} & 442.816 & 424.576 & 458.196 & 480.777  \\ \hline
			KS   & \textbf{0.5787} & 0.0115 & 0.0453 & 0.0222 & 0.0302       \\ \hline
  \end{tabular}
\end{center}
\end{table}

Comparing the empirical survival function with the adjusted distributions it can be observed a better fit for the GWL distribution among the chosen models. These result is confirmed from AIC and AICC since GWL distribution has the minimum values and the p-values returned from the KS test are greater than 0.05. Moreover, considering a significance level of $5\%$, the others models are not able to fit the proposed data. Table 2 displays the MPS estimates, standard errors  and the confidence intervals for $\phi, \lambda$ and $\alpha$ of the GWL distribution. 

\begin{table}[!ht]
\caption{MPS estimates, Standard-error and  $95\%$ confidence intervals (CI) for $\phi, \lambda$ and $\alpha$}
\centering 
\begin{center}
  \begin{tabular}{ c | c |  c| c }
    \hline
		$\boldsymbol{\theta}$  & $\boldsymbol{\hat\theta}_{MPS}$ & S.E($\boldsymbol{\hat\theta}$) & CI$_{95\%}(\boldsymbol{\theta})$ \\ \hline
    \ \ $\phi$ \ \       &  0.0057 & 0.00091 &  ( 0.0039; 0.0075)  \\ \hline
    \ \ $\lambda$   \ \  & 0.0118 & 0.00003 &  (0.0117;  0.0118) \\ \hline
		\ \ $\alpha$   \ \   & 110.4964 & 6.58144 &  (97.5971; 123.3958) \\ \hline
  \end{tabular}
\end{center}
\end{table}

\subsection{Average flows data}

The study of average flows has been proved of high importance to protect and maintain aquatic resources in streams and rivers (Reiser et al., 1989). In this section, we consider a real data set related to the average flows (m$^3$/s) of the Cantareira system during January at S\~{a}o Paulo city in Brazil. Its worth mentioning that the Cantareira system provide water to $9$ million people in the S\~{a}o Paulo metropolitan area. The data set available in Table \ref{cantdaset} was obtained from the website of the National Water Agency including a period from 1930 to 2012.

\begin{table}[ht]
\caption{January average flows (m$^3$/s) of the Cantareira system.}
\centering 
\begin{tabular}{c c c c c c c c c c c} 
\hline 
82.0 & 80.9 & 102.5 & 65.3 & 65.5 & 47.1 & 53.0 & 139.4 & 82.4 & 80.2 & 92.5 \\ 
50.0 & 50.4 & 50.2 & 36.2 & 35.9 & 100.0 & 94.2 & 78.1 & 54.8 & 86.9 & 80.1 \\ 
60.3 & 26.9 & 48.5 & 51.0 & 51.1 & 84.5 & 76.9 & 69.4 & 77.3 & 109.2 & 55.3 \\ 
106.3 & 30.5 & 94.2 & 87.3 & 115.0 & 70.0 & 31.3 & 87.1 & 35.9 & 67.7 & 55.1 \\ 
89.9 & 50.1 & 52.6 & 82.0 & 54.1 & 44.3 & 69.2 & 94.4 & 83.4 & 122.7 & 88.1 \\ 
73.3 & 35.9 & 82.4 & 64.9 & 90.8 & 80.4 & 55.3 & 31.4 & 45.7 & 43.6 & 45.8 \\ 
96.8 & 85.8 & 43.6 & 122.3 & 66.5 & 41.0 & 75.4 & 79.4 & 34.8 & 78.8 & 52.4 \\ 
77.1 & 47.0 & 67.4 & 132.8 & 144.9 & 64.1 \\ [0ex] 
\hline 
\end{tabular}\label{cantdaset}
\end{table}

In this section we consider the ML estimator, showing that both MPS or MLE could be used successfully in applications. Figure \ref{grafico-obscajust2} shows (left panel) the TTT-plot, (middle
panel) the fitted survival superimposed to the empirical survival function and (right panels) the hazard function adjusted by GWL distribution. Table \ref{adjusddc2} presents  the AIC and AICc criteria and the p-value from the KS test for all fitted distributions considering the data set related to the January average flows (m$^3$/s) of the Cantareira system.

\begin{figure}[!htb]
\centering
\includegraphics[scale=0.55]{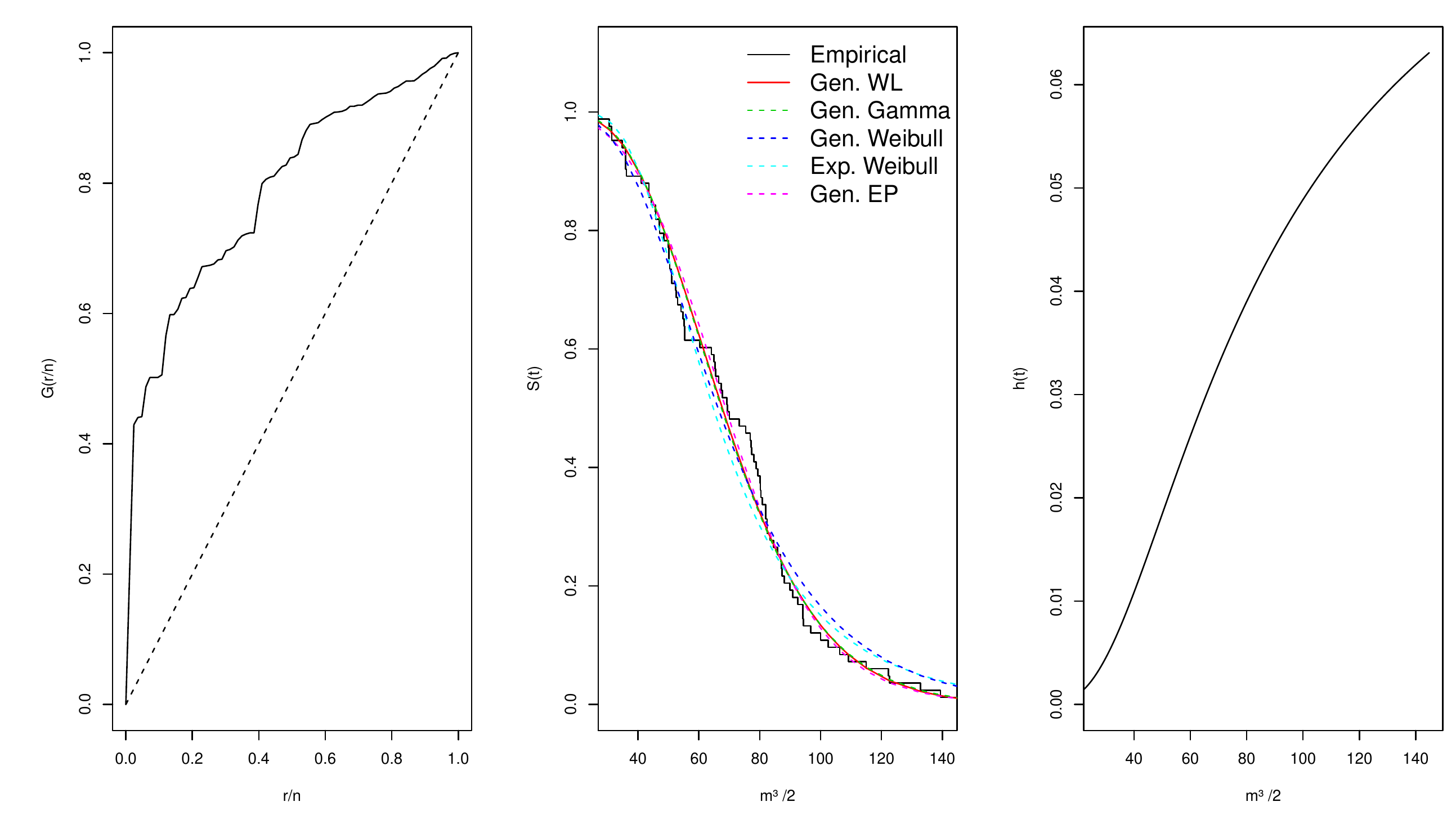}
\caption{(left panel) the TTT-plot, (middle
panel) the fitted survival superimposed to the empirical survival function and (right panels) the hazard function adjusted by GWL distribution}\label{grafico-obscajust2}
\end{figure}

\begin{table}[ht]
\caption{Results of AIC and AICc criteria and the p-value from the KS test for all fitted distributions considering the data set related to the january average flows (m$^3$/s) of the Cantareira system.}
\centering 
\begin{center}
  \begin{tabular}{ c | c | c | c | c | c}
    \hline
		Criteria & Gen. WL & \ Gen. Gamma  \ & \ Gen. Weibull &  \ Exp. Weibull &  \   Gen. EP \\ \hline
      AIC & \textbf{775.431} & 775.461 & 777.280 & 780.304 & 778.873   \\ \hline
			AICc & \textbf{769.735} & 769.765 & 771.584 & 774.608 & 773.176 \\ \hline
			KS   & \textbf{0.4683} & 0.4223 & 0.3935 & 0.1654 & 0.4599       \\ \hline
  \end{tabular}\label{adjusddc2}
\end{center}
\end{table}

Comparing the empirical survival function with the adjusted distributions it can be observed a better fit for the GWL distribution among the chosen models. These result is confirmed from AIC and AICC since GWL distribution has the minimum values and the p-values returned from the KS test are greater than 0.05. Table \ref{adjusd2} displays the ML estimates, standard errors  and the confidence intervals for $\phi, \lambda$ and $\alpha$ of the GWL distribution. 

\begin{table}[!ht]
\caption{ML estimates, Standard-error and  $95\%$ confidence intervals (CI) for $\phi, \lambda$ and $\alpha$}
\centering 
\begin{center}
  \begin{tabular}{ c | c |  c| c }
    \hline
		$\boldsymbol{\theta}$  & $\boldsymbol{\hat\theta}_{MLE}$ & S.E($\boldsymbol{\hat\theta}$) & CI$_{95\%}(\boldsymbol{\theta})$ \\ \hline
    \ \ $\phi$ \ \       &  7.0485 & 1.5425 &  (2.3847; 11.7124)  \\ \hline
    \ \ $\lambda$   \ \  & 0.1244 & 0.0557 &  (0.1183;  0.1305) \\ \hline
		\ \ $\alpha$   \ \   &  0.9579 & 0.1173 &  (0.9310; 0.9849) \\ \hline
  \end{tabular}\label{adjusd2}
\end{center}
\end{table}

\section{Concluding remarks}
\vspace{0.3cm}

In this paper, we propose a new lifetime distribution. The GLW distribution is a straightforwardly generalization of the weighted Lindley distribution proposed by Ghitany et al. (2011), which accommodates increasing, decreasing, decreasing-increasing-decreasing, bathtub, or unimodal hazard functions, making the GWL distribution a flexible model for reliability data. The mathematical properties of the new distribution are discussed. It was also derived the estimation of the parameters of the GWL distribution using eight estimation methods and compared via an intensive simulation study. Most important, from our simulations we observe that the MLE, ME, LSE, WLSE and the CME estimators fail in finding the parameters estimates for a significant number of samples. The simulations show that the MPS (maximum product of spacing) is the most efficient method for estimating the parameters of the GWL distribution in comparison with its competitors. Finally, we analyze two data sets for illustrative purposes, proving that the GWL outperform several usual three parameters lifetime distributions.



\section*{Appendix}

\begin{equation*}
I_{\phi,\phi}=-E\left[\frac{\partial l(\boldsymbol{\theta};\boldsymbol{t})}{\partial\phi^2} \right]= -\frac{1}{(\lambda+\phi)^2}+\psi'(\theta) \ \ \ \ \ \ \ \ \ \ \ \ \ \ \ \ \ \ \ \ \ \ \ \ \ \ \ \ \ \ \ \ \ \ \ \ \ \	\ \ \  
\end{equation*}
\begin{equation*}
I_{\phi,\lambda}=-E\left[\frac{\partial l(\boldsymbol{\theta};\boldsymbol{t})}{\partial\phi\partial\lambda} \right]= -\frac{\alpha}{\lambda}+\frac{1}{(\lambda+\phi)^2}  \ \ \ \ \ \ \ \ \ \ \ \ \ \ \ \ \ \ \ \ \ \ \ \ \ \ \ \ \ \ \ \ \ \ \ \ \ \ \ \ \ \ \ \ \ 
\end{equation*}
\begin{equation*}
I_{\phi,\alpha}=-E\left[\frac{\partial l(\boldsymbol{\theta};\boldsymbol{t})}{\partial\phi\partial\alpha} \right]= \frac{-\alpha\log(\lambda)-\psi(\phi)+\alpha\log(\lambda)-(\lambda+\phi)^{-1}}{\alpha} \ \ \ \ \ \ \ \ \ \ \ \ 
\end{equation*}
\begin{equation*}
\begin{aligned}
I_{\lambda,\lambda}=&-E\left[\frac{\partial l(\boldsymbol{\theta};\boldsymbol{t})}{\partial\lambda^2} \right]= \frac{\alpha\phi}{\lambda^2} +(\alpha-1)\lambda^{\alpha-2}(\psi(\phi)-\alpha\log(\lambda)+ (\lambda+\phi)^{-1}) \\ & +E\left[\frac{\alpha T^\alpha\lambda^{\alpha-2}\left((\alpha-2)\lambda-(\lambda T)^\alpha \right)}{\left(\lambda+(\lambda T)^\alpha \right)}\right]-\frac{1}{(\lambda+\phi)^2} \ \ \ \ \ \ \ \ \ \ \ \ \ \ \ \ \ \ \ \ \ \ \ \ \ \ \ 
\end{aligned}
\end{equation*}
\begin{equation*}
\begin{aligned}
I_{\alpha,\alpha}=&-E\left[\frac{\partial l(\boldsymbol{\theta};\boldsymbol{t})}{\partial\alpha^2} \right]= \frac{\phi(\lambda+\phi+1)\left(\psi(\phi)^2+\psi(\phi)\right)}{\alpha^2(\lambda+\phi)} +\frac{1}{\alpha^2}\\ &  +\frac{2(\lambda+2\phi+1)\psi(\phi)+2}{\alpha^2(\lambda+\phi)}-E\left[\frac{\lambda(\lambda T)^\alpha\log(\lambda T)^2}{\left(\lambda+(\lambda T)^\alpha \right)}\right] \ \ \ \ \ \ \ \ \ \ \ \ \ \ \ \ \ \ \ \ \ \ \ \
\end{aligned}
\end{equation*}
\begin{equation*}
\begin{aligned}
I_{\alpha,\lambda}=&-E\left[\frac{\partial l(\boldsymbol{\theta};\boldsymbol{t})}{\partial\alpha\partial\lambda} \right]= -\frac{\phi}{\lambda}+\frac{\lambda\left(1+\phi\psi(\phi)\right)+\phi\left(1+ (\phi+1)\psi(\phi+1)\right)}{\lambda(\lambda+\phi)} \\& -E\left[\frac{\left(1+\alpha\lambda^{\alpha-1}T^\alpha\right)(\lambda T)^\alpha\log(\lambda T)}{\left(\lambda+(\lambda T)^\alpha \right)^2} \right] +\frac{\left(\phi+\lambda+1-\frac{1}{\alpha}\right)\Gamma\left(\phi+1-\frac{1}{\alpha} \right)}{(\lambda+\phi)\Gamma(\phi)} \\&  -E\left[\frac{\left(\alpha\lambda^{\alpha-1}T^\alpha\log(\lambda T)+ (\lambda T)^{\alpha-1} \right)}{\left(\lambda+(\lambda T)^\alpha \right)} \right]
\end{aligned}
\end{equation*}

\bibliographystyle{abbrv}

\end{document}